\documentclass[11pt,a4,fleqn]{article}
\usepackage{graphicx}
\usepackage{amsmath,amssymb,latexsym,graphics,epsfig}
\usepackage{hyperref}
\usepackage{color}
\usepackage{amsthm}

\newtheorem{theorem}{\bf Theorem}[section]

\newtheorem{lemma}[theorem]{\bf Lemma}
\newtheorem{corollary}[theorem]{\bf Corollary}

\newcommand{\Z}{{\mathbb Z}}

\newcommand{\R}{{\mathbb R}}
\newcommand{\E}{\mathbb E}

\newcommand{\G}{{\mathcal{G}}}

\numberwithin{equation}{section}

\begin{document}
\title{{\Large On the size-Ramsey number of cycles}}

\author{R. Javadi$^{\textrm{a}}$, F. Khoeini$^{\textrm{a}}$, G.R. Omidi$^{\textrm{a},\textrm{c},1}$, A. Pokrovskiy$^{\textrm{b},2}$ \\[2pt]
{\small $^{\textrm{a}}$Department of Mathematical Sciences, Isfahan University of Technology},\\
{\small Isfahan, 84156-83111, Iran}\\
{\small $^{\textrm{b}}$Department of Mathematics, ETH Z\"urich, Switzerland}\\
{\small $^{\textrm{c}}$School of Mathematics, Institute for Research in Fundamental Sciences (IPM),}\\
{\small P.O. Box 19395-5746, Tehran, Iran }\\[2pt]
{rjavadi@cc.iut.ac.ir, romidi@cc.iut.ac.ir,} \\{f.khoeini@math.iut.ac.ir,
dr.alexey.pokrovskiy@gmail.com}}

\date{}

\maketitle \footnotetext[1] {This research is partially
carried out in the IPM-Isfahan Branch and in part supported
by a grant from IPM (No. 95050217).} \vspace*{-0.5cm}

\footnotetext[2] {This research is partially supported
by SNSF grant 200021-149111..} \vspace*{-0.5cm}

\begin{abstract}
For given graphs $G_1,\ldots,G_k$, the size-Ramsey number $\hat{R}(G_1,\ldots,G_k)$ is the smallest integer $m$ for which there exists a graph $H$ on $m$ edges such that in every $k$-edge coloring of $H$ with colors $1,\ldots,k$, $ H $ contains a monochromatic copy of $G_i$ of color $i$ for some $1\leq i\leq k$. We denote $\hat{R}(G_1,\ldots,G_k)$ by $\hat{R}_{k}(G)$ when $G_1=\cdots=G_k=G$. 

Haxell, Kohayakawa and \L{}uczak showed that the size Ramsey number of a cycle $C_n$ is linear in $n$ i.e.  $\hat{R}_{k}(C_{n})\leq c_k n$ for some constant $c_k$. Their proof, is based on the regularity lemma of Szemer\'{e}di and so no specific constant $c_k$ is known.

In this paper, we give various upper bounds for the size-Ramsey numbers of cycles. We give an alternative proof of  $\hat{R}_{k}(C_{n})\leq c_k n$, avoiding the use of the regularity lemma. 
For two colours, we show that for sufficiently large $n$ we have   $\hat{R}(C_{n},C_{n}) \leq 10^6\times cn,$ where $c=843$ if $n$ is even and $c=113482$ otherwise.
\\

\noindent{\small Keywords: Ramsey number, Size Ramsey number, Random graphs, Cycles.}\\
{\small AMS subject classification: 05C55, 05D10}

\end{abstract}

\section{Introduction}
For given graphs $G_1,\ldots,G_k$ and a graph $ H $, we say that $H$ is {\it Ramsey} for $(G_1,\ldots,G_k)$ and we write $H\longrightarrow (G_1,\ldots,G_k)$, if no matter how one colors the edges of $H$ with $k$ colors $1,\ldots,k$, there exists a monochromatic copy of $G_i$ of color $i$ in $ H $, for some $1\leq i\leq k$. Ramsey's theorem \cite{Ramsey} states that for given graphs $G_1,\ldots,G_k$, there exists a graph $H$ that is Ramsey for $(G_1,\ldots,G_k)$. Note that, if a graph $H$ is Ramsey for $(G_1,\ldots,G_k)$ and $H$ is a subgraph of $H'$, then $H'$ is also Ramsey for $(G_1,\ldots,G_k)$. In this view, in order to study the collection of graphs which are Ramsey for $(G_1,\ldots,G_k)$, it suffices to study the collection $ \mathcal{F}(G_1,\ldots,G_k)$ of graphs which are minimal subject to being Ramsey for $(G_1,\ldots,G_k)$. These graphs are called \textit{Ramsey minimal} for $(G_1,\ldots,G_k)$. 

Many interesting problems in graph theory concern the study of various parameters related to Ramsey minimal graphs for $(G_1,\ldots,G_k)$. The most well-known and well-studied one, is the smallest number of vertices of a graph in $\mathcal{F}(G_1,\ldots,G_k)$ which is referred to as the {\it Ramsey number} of $(G_1,\ldots,G_k)$ and is denoted by $R(G_1,\ldots,G_k)$. In diagonal case, where $G=G_1=\cdots=G_k$, we may write $ R_k(G) $ for
$R(G_1,\ldots,G_k)$. Estimating $ R(K_{n})=R_2(K_{n})$ is one of the main open problems in Ramsey theory. Erd\H{o}s \cite{Erdos2} and Erd\H{o}s and Szekeres \cite{Erdos and Szekeres} showed that $2^{n/2}\leq R(K_{n})\leq 2^{2n}$, and despite a lot of work, there have not been improvements to the exponent.
For further results about the Ramsey numbers of graphs, see \cite{recent devel, small Ramsey} and the references therein.

In this paper, we consider another well-studied parameter called the {\it size Ramsey number} $ \hat{R}(G_1,\ldots,G_k)$ of the given graphs $ G_1,\ldots,G_k$, which is defined as the minimum number of edges of a graph in $ \mathcal{F}(G_1,\ldots,G_k)$. The investigation of the size Ramsey numbers of graphs was initiated by Erd\H{o}s et al. \cite{3} in $1978$. Since then, the size Ramsey numbers of graphs have been studied with particular focus on the case of trees, bounded degree graphs and sparse graphs. The survey paper due to Faudree and Schelp \cite{9mir} collects some results about size Ramsey numbers.

One of the most studied directions in this area is the size-Ramsey number of paths. In 1983 Beck~\cite{Be1}, showed that $\hat{R}(P_n)=\hat{R}(P_n,P_n)<900\,n$ for sufficiently large $n$, where $P_n$ is a path on $n$ vertices. This verifies the linearity of the size-Ramsey number of paths in terms of the number of vertices and from then, different approaches were attempted by several authors to reduce the upper bound for $\hat{R}(P_n)$, see \cite{Bo1, D1, L1}. Most of these approaches are based on the classic models of random graphs. Currently, the best known upper bound is due to Dudek et al. \cite{D2} which proved that $ \hat{R}(P_n)\leq 74\, n $, for sufficiently large $ n $.

In this paper, we investigate the size Ramsey number of cycles. The linearity of $\hat{R}_k(C_{n})$ (in terms of $n$) follows from the earlier result by Haxell, Kohayakawa and \L{}uczak~\cite{P. H}. Their proof is based on the regularity lemma and has no specific constant coefficient given.
The following theorem is proved in this paper. 

\begin{theorem}\label{BasicTheorem}
Let $n_1, n_2,  \ldots, n_t$ be a sequence of sufficiently large integers with $t_e$ even numbers and $t_o$ odd numbers. Let $c=82\times 35^{2^{t_o}-2}\times 81^{t_e}$, $n=\max(n_1, \dots, n_t)$ and suppose that for all $i$, we have $n_i\geq 2\lceil \log (nc) \rceil+2$. Then
$$\hat{R}(C_{n_1},\ldots,C_{n_t})\leq  (\ln c+1)\,c^2\, n.$$
\end{theorem}

The above theorem is proved by showing that an Erd\H{o}s-Renyi random graph with suitable edge probability is  almost surely a Ramsey graph for a collection of cycles. By considering other random graph models we will give further improvements on the bounds in Theorem~\ref{BasicTheorem} (see Theorems~\ref{F9}, \ref{F10}, and \ref{F11}.)

Throughout the paper, the notations $\log x$ and $\ln x$ refer to the logarithms to the bases 2 and Euler's number $e$, respectively.


\section{Cycles versus a complete bipartite graph}
In this section, we prove some auxiliary results which will later be used to bound the size Ramsey numbers of cycles. Specifically, we prove some linear upper bounds (in terms of the number of vertices) for the Ramsey and bipartite Ramsey numbers of cycles versus a complete bipartite graph.  First, we give some definitions and lemmas.

A rooted tree with at most two children for each vertex is called the \textit{ binary tree}. The \textit{depth} of a vertex in a binary tree $T$ is the distance from the vertex to the root of $T$ and the maximum distance of any vertex from the root is called the \textit{height} of $T$. If a tree has only one vertex (the root), the height is zero. A \textit{ perfect binary tree} is a binary tree with all leaves at the same depth where every internal vertex (non-leaf vertex) has exactly two children. Now we begin with the following lemma.

\begin{lemma}\label{Binary-trees}
For every positive integer $n\geq 2$, there is a binary tree with $n$ leaves at depth $\lceil \log n \rceil$ and at most $2n+\lceil \log n \rceil-2$ vertices.
\end{lemma}
\begin{proof}
If $n=2^{t}$ for some $t$, then clearly the perfect binary tree of height $t$ has exactly $n$ leaves and $2n-1$ vertices and we are done. Now, assume that $n=2^{t_1}+\cdots +2^{t_r}$, where $r\geq 2$ and $t_1> \cdots> t_r\geq0$. For each $1\leq i\leq r$, let $T_i$ be the perfect binary tree of height $t_i$ with $2^{t_i+1}-1$ vertices and $2^{t_i}$ leaves. Now, we construct a binary tree $T$ as follows. Consider the vertex disjoint binary trees $T_1,\ldots,T_r$ with  roots $x_1,\ldots,x_r$ and a new path $P=v_1\ldots v_{t_1-t_r+1}$ and add an edge from $v_{t_1-t_i+1}$ to the root $x_i$ of $T_i$ for each $ 1\leq i\leq r $. One can easily check that $T$ is a binary tree rooted at $ v_1 $  with $n=2^{t_1}+\cdots +2^{t_r}$ leaves at depth $t_1+1$ and
$$|V(T)|=\sum_{i=1}^{r}2^{t_i+1}-r+t_1-t_r+1\leq 2n+t_1-1$$
vertices. Clearly $\lceil \log n \rceil=t_1+1$ and so $T$ is a binary tree with $n$ leaves and at most $2n+\lceil \log n \rceil-2$ vertices.
\end{proof}

We also need the following theorem due to Friedman and Pippenger.
\begin{theorem}{\rm \cite{FP}}\label{c}
Let $n$ and $d$ be positive integers and let $G$ be a non-empty graph such that for every $X\subseteq V(G)$ with $|X|\leq 2n-2$,
$|N_{G}(X)|\geq (d+1)|X|$, where $N_{G}(X)$ is the set of all vertices in $V(G)$ which are adjacent (in $G$) to a vertex in $X$. Then $G$ contains every tree with $n$ vertices and maximum degree at most $d$.
\end{theorem}

The above theorem is used to prove the following lemma about finding red paths in a $2$-coloured balanced complete bipartite graph. The proof technique of this lemma is similar to techniques from~\cite{BSP, SP}.

\begin{lemma}\label{LemmaPath}
For given integers $m_1,m_2 $, let $n$ be an even integer with $m_2 \geq n \geq (\lceil \log m_1 \rceil+\lceil \log m_2 \rceil+1)$.
Suppose that we have a $2$-edge-colored  $K_{32m_1+49m_2, 32m_1+49m_2}$ with colors red and blue and bipartition classes $X$ and $Y$ which has no blue $K_{m_1, m_2}$. Then there are $X'\subseteq X$ and $Y'\subseteq Y$ with $|X'|=m_1$ and $|Y'|=m_2$ such that for every $x\in X'$ and $y\in Y'$ there is a red path of length $n-1$ from $x$ to $y$.
\end{lemma}
\begin{proof}
Assume that the edges of $H=K_{|X|,|Y|}$ are colored by red and blue and $(X,Y)$ is the bipartition of $H$ with
$|X|=|Y|=32m_1+49m_2$.  Let $H_r$ and $H_b$ be the subgraphs of $H$ induced on the red and blue edges, respectively. By our assumption,  $H_b$ is $K_{m_1,m_2}$-free. Thus, we have $|N_{H_r}(S)|\geq 32m_1+48m_2$, for every $S\subseteq X$ or $ Y $, with $|S|\geq m_1$. In particular, this implies that

\begin{align}\label{H}
|N_{H_r}(S)|\geq 32m_1+48m_2, \quad \text{ for every } S\subseteq V(G)\quad  \text{with} \quad |S|\geq 2m_1.
\end{align}
This holds since for any such set $ S $, either $|S\cap X|\geq m_1$ or  $|S\cap Y|\geq m_1$.\\

\noindent \textbf{Claim 1.}
There is an induced subgraph $G\subseteq H_r$ which satisfies
\begin{align}\label{2}
|N_{G}(S)|\geq 4|S|, \quad \text{ for every } S\subseteq V(G)\quad  \text{with} \quad |S|\leq 4m_1+6m_2.
\end{align}
\begin{proof}[Proof of the claim.]
Let $A$ be the largest subset of $V(H_r)$ with $|N_{H_r}(A)|< 4|A|$ and $|A|\leq 4m_1+6m_2.$ Note that if there is no such a subset $A$, then $G=H_r$ has the desired property. Now let $G$ be obtained from $ H_r $ by removing vertices in $ A $. To see \eqref{2}, let $S\subseteq V(G)$ be a subset with $|S|\leq 4m_1+6m_2$. For the contrary, suppose that
$|N_{G}(S)|< 4|S|$. Then
$$|N_{H_r}(S\cup A)|=|N_{H_r}(S) \cup N_{H_r}(A)|\leq |N_{G}(S)|+|N_{H_r}(A)|< 4|S|+4|A|=4|S\cup A|.$$
By maximality of $A$, we have $|S\cup A|> 4m_1+6m_2$. But then $|N_{H_r}(S\cup A)|< 4|S\cup A|\leq 32m_1+48m_2$.  This contradicts (\ref{H}).
\end{proof}
Now, let $ G $ be the subgraph of $ H_r $ which satisfies \eqref{2}. By using  Theorem~\ref{c}, $G$ (and so $H_r$) contains a copy of any tree $T$ with at most $2m_1+3m_2+1$ vertices and maximum degree at most $3$. Now for $i=1,2$, let $T_i$ be a binary tree  with $m_i$ leaves   at depth $\lceil \log m_i \rceil$ and at most $2m_i+\lceil \log m_i \rceil-2$ vertices (which exists due to Lemma~\ref{Binary-trees}). Also, let $T$ be a tree on at most $n+2m_1+2m_2$ vertices formed by attaching the roots of $ T_1 $ and $ T_2 $ by a path of length $n-1-\lceil \log m_1 \rceil-\lceil \log m_2 \rceil$. Note that $T$ has maximum degree $3$ with leaves $x_{1},\ldots,x_{m_1}$, $y_{1},\ldots,y_{m_2}$, where there is a path of length $n-1$ from $x_{i}$ to $y_{j}$ for every $1\leq i\leq m_1$ and $1\leq j\leq m_2$. Also note that since $n$ is even, $\{x_{1},\ldots,x_{m_1}\}$ and $\{y_{1},\ldots,y_{m_2}\}$ are contained in different parts of the bipartition of $T$. By Theorem~\ref{c}, $H_r$ contains a copy of $T$. Without loss of generality, we can assume that $\{x_{1},\ldots,x_{m_1}\}\subseteq X$ and  $\{y_{1},\ldots,y_{m_2}\}\subseteq Y$. The sets $X'=\{x_{1},\ldots,x_{m_1}\}$ and $Y'=\{y_{1},\ldots,y_{m_2}\}$ satisfy the requirements of the lemma.
\end{proof}

Given bipartite graphs $G_1,\ldots,G_k$, the {\it bipartite Ramsey number} $BR(G_1,\ldots,G_k)$ is defined as the smallest integer $b$ such that for any edge coloring of the complete bipartite graph $K _{b,b}$ with $k$ colors $1,\ldots,k$, there exists a monochromatic copy of $G_i$ of color $i$ in $ K_{b,b} $, for some $1\leq i\leq k$. In other words, it is the smallest integer $ b $ such that $K _{b ,b} \rightarrow (G_1,\ldots,G_k)$. The above lemma can be used to give an upper bound for the bipartite Ramsey number of an even cycle versus a complete bipartite graph.
\begin{lemma}\label{F1}
For given integers $ m_1,m_2 $, let $n$ be an even integer with $m_2 \geq n \geq (\lceil \log m_1 \rceil+\lceil \log m_2 \rceil+1)$. Then, $BR(C_{n},K_{m_1,m_2})\leq 32m_1+49m_2$.
\end{lemma}
\begin{proof}
Let $H$ be a $2$-edge-colored complete bipartite graph with parts $X$ and $Y$ such that $|X|=|Y|=32m_1+49m_2$. Suppose that H contains no blue $K_{m_1, m_2}$. To prove the lemma, it is enough to show that $H$ contains a red $C_n$.

By Lemma~\ref{LemmaPath} there are $X'\subseteq X$ and $Y'\subseteq Y$ with $|X|=m_1$ and $|Y|=m_2$ such that for every $x\in X'$ and $y\in Y'$ there is a red path of length $n-1$ from $x$ to $y$. Since $H$ has no blue $K_{m_1, m_2}$, there is a red edge $xy$ for some $x\in X'$ and $y\in Y'$. Adding this edge to the red path of length $n-1$ from $x$ to $y$ gives a red cycle of length $n$ as required.
\end{proof}

The following corollary is an immediate consequence of Lemma~\ref{F1}.

\begin{corollary}\label{F0}
Let $ m $ and $ n_1,\ldots, n_t $ be positive integers such that for every $1\leq i\leq t  $, $ n_i $ is even and $m\geq n_i\geq 2 \lceil \log (81^{t-1}m)\rceil+1$. Then, $BR(C_{n_1},\ldots,C_{n_t},K_{m,m})\leq 81^{t}m$.
\end{corollary}

\begin{proof}
We give a proof by induction on $t$. The case $t=1$ follows from Lemma \ref{F1}. Now, assuming the assertion holds  for $t<t_0$, we are going to prove it for $t=t_0$. To see this, consider the graph $H=K_{81^{t_0}m,81^{t_0}m}$ whose edges are colored by the colors $1,2,\ldots,t_0+1$ and suppose that there is no copy of $C_{n_i}$ of color $i$ in $ H $ for all $1\leq i\leq t_0$. We show that there is a copy of $K_{m,m}$ of color $t_0+1$ in $ H $. By the induction hypothesis, we have $BR(C_{n_1},\ldots,C_{n_{t_0-1}},K_{81m,81m})\leq 81^{t_0}m$ and so there is a copy of $K_{81m,81m}$ in $ H $ whose edges are colored by the colors $t_0$ and $t_0+1$. Now using Lemma \ref{F1}, this copy contains a copy of $C_{n_{t_0}}$ of color $t_0$, or a copy of $K_{m,m}$ of color $t_0+1$. The earlier case does not hold, since $ H $ has no copy of $C_{n_{t_0}}$ of the color $t_0$. Hence, there is a copy of $K_{m,m}$ of the color $t_0+1$ in $ H $ and we are done.
\end{proof}

The proof of the following useful lemma is similar to the proof of Lemma~\ref{LemmaPath}, so we omit the proof.

\begin{lemma}\label{LemmaPath1}
For given integers $m_1,m_2 $, let $n$ be an integer with $m_2 \geq n \geq (\lceil \log m_1 \rceil+\lceil \log m_2 \rceil+1)$.
Suppose that we have a $2$-edge-colored  $K_{33m_1+49m_2}$ with colors red and blue which has no blue $K_{m_1, m_2}$. Then there are two disjoint subsets of vertices $X$ and $Y$ of sizes $m_1$ and $m_2$ such that for every  $x\in X$ and $y\in Y$ there is a red path of length $n-1$ from $x$ to $y$.
\end{lemma}

The following result follows from Lemma~\ref{LemmaPath1} along with the same argument as we used in the proof of Lemma \ref{F1}.

\begin{lemma}\label{F2}
Let $ n,m_1,m_2 $ be positive integers, where $m_2 \geq n \geq (\lceil \log m_1 \rceil+\lceil \log m_2 \rceil+1)$. Then, $R(C_{n},K_{m_1,m_2})\leq 33m_1+49m_2$.
\end{lemma}


We also need the following lemma.

\begin{lemma}\label{F4}
Let $ n,m_1,m_2 $ be positive integers, where $m_2 \geq n \geq (\lceil \log m_1 \rceil+\lceil \log m_2 \rceil+2)$.  Then
$$K_{X,Y,Z}\longrightarrow (C_n,K_{m_1,m_2}),$$
where $K_{X,Y,Z}$ is a complete 3-partite graph with color classes $X, Y, Z$ of sizes $|X|=|Y|=32m_1+49m_2$ and $|Z|=m_1+m_2-1$.
\end{lemma}
\begin{proof}
The case when $n$ is even follows from Lemma~\ref{F1} (it just suffices to consider the subgraph $K_{X,Y}$ of $K_{X,Y,Z}$ and apply Lemma \ref{F1}). Now, let $n$ be odd. Consider a $ 2- $edge coloring of $K_{X,Y,Z}$  and suppose that there is no blue $K_{m_1, m_2}$.

By Lemma~\ref{LemmaPath}, there are sets $X'\subseteq X$ and $Y'\subseteq Y$ with $|X'|=m_1$ and $|Y'|=m_2$ such that for every $x\in X'$ and $y\in Y'$ there is a red path of length $n-2$ from $x$ to $y$ contained in $X\cup Y$.

Now, since there is no blue $K_{m_1,m_2}$ in the 2-edge-colored $K_{X,Z}$, we have $|N^r_{Z}(X')|\geq m_1$, where $N^r_{Z}(S)$ is the set of all vertices in $Z$ which have a neighbour in $ S $ in the red subgraph of $K_{X,Y,Z}$.
Similarly, since there is no blue $K_{m_1,m_2}$ in the 2-edge-colored $K_{Y,Z}$, we have  $|N^r_{Z}(Y')|\geq m_2$.
Therefore, since $|Z|=m_1+m_2-1$, we have $N^r_{Z}(X')\cap N^r_{Z}(Y')\neq \emptyset$. Hence, there are some vertices $x\in X',$ $y\in Y'$,  and $z\in N^r_{Z}(x)\cap N^r_{Z}(y)$. Now, concatenation of the edges $yz$ and $zx$ and the path of length $n-2$ from $x$ to $y$ in $X\cup Y$ comprises a red $C_n$, as required.
\end{proof}

Let $f_1(m_1,m_2)=33m_1+49m_2$ and for every $t\geq 2$, define

\begin{align}\label{F}
f_t(m_1,m_2)=f_{t-1}(f_{t-1}(32m_1+49m_2,m_1+m_2-1),32m_1+49m_2).
\end{align}
In the following, we show that $f_t(m_1,m_2)$ is an upper bound for the Ramsey number of $t$ cycles (with some restrictions on their sizes) versus the graph $K_{m_1,m_2}$.

\begin{theorem}\label{F5}
Let $ m_1,m_2 $ and $ n_1,\ldots, n_t $ be positive integers such that $m_2\geq n_i\geq 2 \lceil \log (f_t(m_1,m_2))\rceil+2$ for each $1\leq i\leq t$. Then, $$R(C_{n_1},\ldots,C_{n_t},K_{m_1,m_2})\leq f_t(m_1,m_2).$$
\end{theorem}
\begin{proof}
We give a proof by induction on $t$. The case $t=1$ follows from Lemma \ref{F2}. Now, assuming correctness of the assertion for $t<t_0$, we are going to prove it for $t=t_0$. Consider the $(t_0+1)$-edge-colored graph $H=K_{N}$ with colors $1,2,\ldots,t_0+1$, where $N=f_{t_0}(m_1,m_2)$. We assume that $H$ contains no copy of $C_{n_i}$ of color $i$ for each $1\leq i\leq t_0$ and we show that there is a copy of $K_{m_1,m_2}$ of color  $t_0+1$. By the induction hypothesis, we have $R(C_{n_1},\ldots,C_{n_{t_0-1}},K_{N_1,N_2})\leq N$ for $N_1=f_{t_0-1}(32m_1+49m_2,m_1+m_2-1)$ and $N_2=32m_1+49m_2$. So, there is a copy of 2-edge-colored $K_{N_1,N_2}$ with parts $X$ and $Y$ by colors $t_0$ and $t_0+1$ in $K_{N}$. Now, again by the induction hypothesis, we have $$|X|=N_1=f_{t_0-1}(32m_1+49m_2,m_1+m_2-1)\geq R(C_{n_1},\ldots,C_{n_{t_0-1}},K_{32m_1+49m_2,m_1+m_2-1}).$$
Therefore, there is a copy of a 2-edge-colored $K_{32m_1+49m_2,m_1+m_2-1}$ by the colors $t_0$ and $t_0+1$ with parts $X'$ and $X''$  in the induced subgraph of $K_{N}$ on $X$. Thus, the edges of the complete 3-partite graph with the color classes $Y$, $X'$ and $X''$  are colored by the colors $t_0$ and $t_0+1$ and so by Lemma \ref{F4}, there is a copy of $K_{m_1,m_2}$ of color  $t_0+1$ in $ H $ and we are done.
\end{proof}

The following corollary follows from Theorem \ref{F5} and the fact that $f_2(m_1,m_2)=38033m_1+57379m_2-1617$.

\begin{corollary}\label{F21}
For positive integers $m_1,m_2,n_1,n_2$ with $m_2\geq n_1\geq n_2\geq 2\lceil \log (38033m_1+57379m_2-1617)\rceil +2$, we have $$R(C_{n_1},C_{n_2},K_{m_1,m_2})\leq 38033m_1+57379m_2-1617.$$
\end{corollary}

By calculating the function $f_t(m_1,m_2)$ and using Theorem~\ref{F5}, we can prove the following theorem.

\begin{theorem}\label{F20}
Let $t\geq 2$ and $m_1,m_2$ and $ n_1,\ldots, n_t $ be positive integers such that $m_2\geq n_i\geq 2 \lceil \log (35^{2^{t}-2}(32m_1+49m_2))\rceil+2$ for each $1\leq i\leq t$.
Then, $$R(C_{n_1},\ldots,C_{n_t},K_{m_1,m_2})\leq 35^{2^{t}-2}(32m_1+49m_2).$$
\end{theorem}

\begin{proof}
Using Theorem \ref{F5}, it just suffices to show that $f_t(m_1,m_2)\leq 35^{2^{t}-2}(32m_1+49m_2)$. To see this, let $f_t(m_1,m_2)=a_tm_1+b_tm_2+c_t$, where $a_t$, $b_t$ and $c_t$ are three functions in terms of $t$. From (\ref{F}) one can easily see that for each $t\geq 2$,
$$a_t=32a_{t-1}^2+a_{t-1}b_{t-1}+32b_{t-1},$$ $$b_t=49a_{t-1}^2+a_{t-1}b_{t-1}+49b_{t-1},$$ and $$c_t=-a_{t-1}b_{t-1}+a_{t-1}c_{t-1}+c_{t-1}.$$
Clearly for every $i\geq 2$ we have $a_i<b_i<49/32a_i$ and so for every $ t\geq 2 $, $$a_t<32a_{t-1}^2+\frac{49}{32}a_{t-1}^2+49a_{t-1}\leq 35a_{t-1}^2.$$
Therefore, by induction on $t$ we can see that for every $ t\geq 2 $, we have $$a_t\leq32(35^{2^{t}-2}),$$ and hence $$b_t\leq 49(35^{2^{t}-2}).$$
On the other hand, again by induction on $t$, we have $c_t\leq 0$. Therefore $$f_t(m_1,m_2)\leq a_tm_1+b_tm_2\leq 35^{2^{t}-2}(32m_1+49m_2).$$
\end{proof}


With all these results in hand, we can prove the main result of this section, as follows.

\begin{theorem}\label{F6}
Let $t_e$ and $t_o$ be respectively the number of even and odd integers in the sequence $(n_1,\ldots,n_t)$ and suppose that $m\geq n_i\geq 2(\lceil \log N \rceil+1)$ for each $1\leq i\leq t$, where $N=82\times35^{2^{t_o}-2}\times 81^{t_e}m$.  Then $$R(C_{n_1},\ldots,C_{n_t},K_{m,m})\leq N.$$
\end{theorem}

\begin{proof}
The case $t_o=0$ follows from Corollary~\ref{F0}. So, let $t_o\geq 1$.  Also, without loss of generality, assume that $n_i$ is odd for all $1\leq i\leq t_o$.
Consider a $(t+1)$-edge-colored  $K_N$ with colors $1,2,\ldots,t+1$.  Assume that there is no copy of $C_{n_i}$ of color $i$ for each  $1\leq i\leq t$. Our goal is to show that there is a copy of $K_{m,m}$ of color $t+1$.  Using Lemma~\ref{F2} when $t_o=1$ and Theorem \ref{F20} when $t_o\geq 2$, we have $R(C_{n_1},\ldots,C_{n_{t_o}},K_{81^{t_e}m,81^{t_e}m})\leq  N$ and so there is a copy of $K_{81^{t_e}m,81^{t_e}m}$ in $K_N$ whose edges are  colored by $t_e+1$ colors $t_o+1,\ldots,t+1$. Now, Corollary \ref{F0} implies that $BR(C_{n_{t_o+1}},\ldots,C_{n_t},K_{m,m})\leq 81^{t_e}m$ and so there is a copy of $K_{m,m}$ of color $t+1$ in the $(t_e+1)$-edge-colored $K_{81^{t_e}m,81^{t_e}m}$, as desired.
\end{proof}


\section{Random graphs and upper bounds}
In this section, we will apply the obtained results in Section 2 on random graphs to give some linear upper bounds in terms of the number of vertices for the size Ramsey number of large cycles. For this purpose, we deploy three random structure models namely random graphs, random regular graphs and random bipartite graphs. First, let us recall a classic model of random graphs that will be applied in this section. The {\it binomial random graph} $\G(n,p)$ is the random graph $G$ with the vertex set $[n]:=\{1,2, \ldots, n\}$ in which every pair $\{i,j\} \subseteq [n]$ appears independently as an edge in $G$ with probability~$p$. For two subsets of vertices $ S,T $, the number of edges with one end in $ S $ and one end in $ T $ is denoted by $ e(S,T) $. Recall that an event in a probability space holds \emph{asymptotically almost surely} (or \emph{a.a.s.}) if the probability that it holds tends to $1$ as $n$ goes to infinity. To see more about random graphs we refer the reader to see \cite{Bo1,AS}. We will state some results that hold a.a.s. and we always assume that $n$ is large enough.
The first lemma asserts that there is a graph on $ N $ vertices whose number of edges is linear in terms of $ N $, while it has no large hole (a pair of disjoint subsets of vertices with no edge between them). It should be noted that similar ideas have been used in \cite{D1, DKP} to prove  linear upper bounds for  size Ramsey numbers.

\begin{lemma}\label{lem:random1}
Let $c \in \R_+$ and let $d=d(c)$ be such that
\begin{equation}
\label{eqcd}
(1-2c)\ln(1-2c) + 2c \ln(c)+c^2d\geq 0.
\end{equation}
Then, in the graph $G \in \G(N,d/N)$, a.a.s. for every two disjoint sets of vertices $S$ and $T$ with $|S|=|T|=cN$,  we have $e(S,T) \geq 1$.
\end{lemma}
\begin{proof}
Let $S$ and $T$ with $|S| = |T| = cN$ be fixed and let $X=X_{S,T}=e(S,T)$.
Clearly, $\E X = c^{2}dN > 0$ and by the Chernoff bound
\begin{eqnarray*}
\Pr(X = 0)=\Pr(X \le 0) &\le& \exp\left( - \E X  \right)=\exp \Big( -c^{2}dN  \Big).
\end{eqnarray*}
Thus, by the union bound over all choices of $S$ and $T$ we have
\begin{align*}
\Pr\left( \bigcup_{S,T} ( X_{S,T} = 0 ) \right) &\le
\binom{N}{cN} \binom{(1-c)N}{cN} \exp \Big( -c^{2}dN \Big) \\
&= \frac{ N!}{(cN)! (c N)! ((1-2c) N)!} \exp \Big( -c^{2}dN \Big).
\end{align*}
Using Stirling's formula ($x! \sim \sqrt{2\pi x} (x/e)^x$) we get
\begin{align*}
\Pr & \left( \bigcup_{S,T} (X_{S,T} = 0) \right)\leq \frac{1}{2\pi c\sqrt{1-2c}N}.\left(\frac{(1-2c)^{2c-1}\exp \Big( -c^{2}d  \Big)}{c^{2c}}\right) ^{ N}\\
& \leq  \frac{1}{2\pi c\sqrt{1-2c}N}=o(1),
\end{align*}
where the last inequality is due to \eqref{eqcd}. This complete the proof. 
\end{proof}

Combining Lemma \ref{lem:random1} and Theorem \ref{F6}, gives some information on the size Ramsey numbers of  cycles. Roughly speaking, these two facts imply that $\G(N,d/N)\longrightarrow (C_{n_1},\ldots,C_{n_t})$ for sufficiently large $N$ when we have some restrictions on the parameters. In the following result we use this fact to give a linear upper bound for the size Ramsey number of large cycles. It is a strengthening of Theorem~\ref{BasicTheorem}.

\begin{theorem}\label{F9}
Let $f=82\times 35^{2^{t_o}-2}\times 81^{t_e}$, where $t_e$ and $t_o$ are respectively the number of even and odd integers in the sequence $(n_1,\ldots,n_t)$. Also let $c=\min\{95412,f\}$ if $t=2$ and $c=f$, otherwise. Suppose that $n=\max\{n_1,\ldots, n_t\} $ and for each $1\leq i\leq t$, we have $n_i\geq 2\lceil \log (nc) \rceil+2$. Then, for sufficiently large $n$, we have $$\hat{R}(C_{n_1},\ldots,C_{n_t})\leq  (\ln c+1)\,c^2\, n.$$
\end{theorem}

\begin{proof}
Let $ N=nc $, $ d= ((2c^{-1}-1)\ln(1-2c^{-1}) - 2c^{-1} \ln(c^{-1}))/c^{-2}$ and $G=\G(N,d/N)$. By Lemma \ref{lem:random1},  a.a.s. for every two disjoint sets of vertices $S$ and $T$ in $V(G)$ with $|S|=|T|=n$, we have $e(S,T) \geq 1$. Therefore, a.a.s. the complement of $G$ does not contain $K_{n,n}$ as a subgraph. 
On the other hand, the expected number of edges of $G$ is  $\frac{d}{N}\binom{N}{2}\leq Nd/2$ and the concentration around the expectation follows immediately from the Chernoff bound. Hence, for sufficiently large $ N $, there exists a graph $ H $ on $ N $ vertices with at most $ Nd/2 $ edges whose complement does not contain $ K_{n,n} $ as a subgraph. Hence, by Corollary \ref{F21} and Theorem \ref{F6}, we have $H\longrightarrow (C_{n_1},\ldots,C_{n_t})$. This means that for sufficiently large $N$ we have
\[\hat{R}(C_{n_1},\ldots,C_{n_t})\leq \frac{Nd}{2}\leq   \frac{c\ln c-(c-2)\ln(c-2) }{2}\, {c^2}\, n\leq  (\ln c+1)\,{c^2}\, n,
\]
where the last inequality follows by applying  the mean value theorem to the function $x\ln x$.
\end{proof}


Based on Theorem \ref{F9}, for sufficiently large $n$, we have
$$\hat{R}(C_{n},C_{n}) \leq
\begin{cases}
113484 \times 10^{6}\, n & \text{if } n \text{ is odd}, \\
2515\times 10^{6}\, n & \text{if } n \text{ is even}.
\end{cases}
$$

Another probability space that we are interested in is the space of random $d$-regular graphs on $ n $ vertices with uniform probability distribution. This space is denoted by $\mathcal{G}_{n,d}$, where $ d\geq 2 $ is fixed and $ n $ is large enough and even when $d$ is odd. Instead of working directly in $\mathcal{G}_{n,d}$, we use the \textit{pairing model} (also known as the \textit{configuration model}) of random regular graphs, first introduced by Bollob\'{a}s~\cite{Bollobas}, which is described here. Suppose that $dn$ is even and consider $dn$ points partitioned into $n$ labeled buckets $v_1,v_2,\ldots,v_n$ each of which contains exactly $d$ points. A \textit{pairing} of these points is a perfect matching into $dn/2$ pairs. Given a pairing $P$, we may construct a $ d $-regular multigraph $G(P)$, with loops and parallel edges allowed, as follows. The vertices are the buckets $v_1,v_2,\ldots, v_n$, and a pair $\{x,y\}$ in $P$ corresponds to an edge $v_iv_j$ in $G(P)$ if $x$ and $y$ are contained in the buckets $v_i$ and $v_j$, respectively. It can be easily seen that the probability that the random pairing yields a given simple graph is uniform, hence the restriction of the probability space of random pairing to simple graphs is precisely $\mathcal{G}_{n,d}$. Moreover, it is well known that a random pairing generates a simple graph with probability asymptotic to $e^{-(d^2-1)/4}$, depending on $d$. Thus, for fixed $ d $, any event holding a.a.s.\ over the probability space of random pairing also holds a.a.s.\ over the corresponding space $\mathcal{G}_{n,d}$. For this reason, asymptotic results over random pairing model are sufficient for our purposes. For more information on this model, see, for instance, the survey of Wormald~\cite{Wormald}.\\\\

Now let

\begin{eqnarray*}
f(a,c,d) &=& g(c)+g(d)+g((c-2)d)+\frac{1}{2}g((c-1-a)d)-g(c-2)-g(ad)\\
&& \quad -g((c-2-a)d)-\frac{1}{2}g((1-a)d)-\frac{1}{2}g(cd),
\end{eqnarray*}
where $g(x)=x\ln (x)$. The following
lemma is the counterpart of Lemma~\ref{lem:random1} for random regular graphs and will be used to give another linear upper bound for the size Ramsey numbers of large cycles.

\begin{lemma}\label{thm:d-reg1}
Let $c \in \R_+$ and $d=d(c)$ be non-negative real numbers such that
$f(a,c,d)\leq 0$ for every real number $a$ with $0\leq a\leq 1$.
Then,  a.a.s. for every two disjoint sets of vertices $S$ and $T$ in $\mathcal{G}_{N,d}$ with $|S|=|T|={N}/{c}$,  we have $e(S,T) \geq 1$.
\end{lemma}
\begin{proof}
Let $G=\mathcal{G}_{N,d}$. Our goal is to show that the expected number of pairs of two disjoint sets, $S$ and $T$ in $G$, such that $|S|=|T|={N}/{c}$ and $e(S,T) =0$ tends to zero as $N \to \infty$. This, together with the first moment principle, implies that a.a.s. no such pair exists.\\
Let $m={N}/{c}$ and $a=a(m)$ be any function of $m$ such that $amd \in \Z$ and $0\le a\le 1$. Let $X(a)$ be the expected number of pairs of two disjoint sets $S, T$ such that $|S|=|T|=m$, $e(S,T) = 0$, and $e(S,V \setminus (S \cup T)) = adm$. Using the pairing model, it is clear that
\begin{eqnarray*}
X(a) &=& {N \choose m} {N-m \choose m} {Nd-2md \choose adm}{md \choose adm} (amd)! \\
&& \quad \cdot \ M(md-amd) \cdot M \Big( Nd-md-amd \Big) \Big/ M(Nd),
\end{eqnarray*}
where $M(i)$ is the number of perfect matchings on $i$ vertices, that is,
$$
M(i) = \frac {i!} {(i/2)! 2^{i/2}}.
$$
After simplification we get
\begin{eqnarray*}
X(a) &=& (N)! (md)! (Nd-md-amd)! (Nd/2)! (Nd-2md)! 2^{amd} \\
&& \quad \cdot \ \Bigg[ (m!)^2 (adm)! (N-2m)! (Nd-2md-adm)! ((md-adm)/2)!  \\
&& \quad \quad \quad (Nd)!  ((Nd-md-amd)/2)! \Bigg]^{-1}.
\end{eqnarray*}
Using Stirling's formula ($i! \sim \sqrt{2\pi i} (i/e)^i$) and focusing on the exponential part we obtain
$$X(a) = c_me^{f(a,c,d)m},$$
where $c_m=o(1)$ and
\begin{eqnarray*}
f(a,c,d) &=& c \ln c + d \ln d +(cd-2d)\ln(cd-2d) + (cd-d-ad)/2 \ln(cd-d-ad)\\
&& \quad -(c-2) \ln (c-2)-ad \ln(ad)-(cd-2d-ad) \ln(cd-2d-ad) \\
&& \quad - (d-ad)/2 \ln (d-ad)-cd/2\ln(cd).
\end{eqnarray*}
Our assumptions imply that $f(a,c,d) \le 0$  for any integer $adm$ under consideration. Then we would get $\sum_{adm} X(a) = o(1)$ (as $adm = O(m)$). Thus, the desired property is satisfied and this completes the proof.
\end{proof}

With the same assumptions as in Theorem \ref{F6} and Lemma \ref{thm:d-reg1} for $c=35^{2^{t_o}-2}81^{t_e+1}$ we can conclude that  $\mathcal{G}_{N,d}\longrightarrow (C_{n_1},\ldots,C_{n_t})$ for sufficiently large $N$. In the following theorem we use this fact to give a linear upper bound for the size Ramsey numbers of large cycles.

\begin{theorem}\label{F10}
Let $f=82\times 35^{2^{t_o}-2} \times 81^{t_e}$, where $t_e$ and $t_o$ are respectively the number of even and odd integers in the sequence $(n_1,\ldots,n_t)$. Also let $c=\min\{95412,f\}$ if $t=2$ and $c=f$, otherwise. Suppose that $n=\max\{n_1,\ldots, n_t\}$. Also suppose that for each $1\leq i\leq n$ we have $n_i\geq 2\lceil \log N \rceil+2$, where $N=cn$. Let $d=d(c)$ be a non-negative real number such that
$f(a,c,d)\leq 0$ for every real number $a$ with $0\leq a\leq 1$. Then for sufficiently large $N$ we have $$\hat{R}(C_{n_1},\ldots,C_{n_t})\leq Nd/2.$$
\end{theorem}

\begin{proof}
Consider $G=\mathcal{G}_{N,d}$. By Lemma \ref{thm:d-reg1}, a.a.s. for every two disjoint sets of vertices $S$ and $T$ in $V(G)$ with $|S|=|T|={N}/{c}$, we have $e(S,T) \geq 1$. Therefore a.a.s. the complement of $G$ does not contain $K_{N/c,N/c}$ as a subgraph and so by Corollary \ref{F21} and Theorem \ref{F6}, we have $G\longrightarrow (C_{n_1},\ldots,C_{n_t})$.
On the other hand, the expected number of edges of $G$ is  $Nd/2$ and the concentration around the expectation follows immediately from Chernoff bound. This means that for sufficiently large $N$ we have $$\hat{R}(C_{n_1},\ldots,C_{n_t})\leq Nd/2.$$
\end{proof}

Now, we use Theorem \ref{F10} to give an upper bound for $\hat{R}(C_{n},C_{n})$. If $n$ is odd (resp. even), then $c=95412$ (resp. $c=538002/35$) and by an easy computation one can verify that for $d=2378778$ (resp. $d=327091$),  we have $f(a,c,d)\leq 0$ for every real number $a$ with $0\leq a\leq 1$. Therefore, for sufficiently large $ n $, we have
$$\hat{R}(C_{n},C_{n}) \leq
\begin{cases}
113482 \times 10^{6}\, n & \text{if } n \text{ is odd}, \\
2514 \times 10^{6} \, n & \text{if } n \text{ is even}.
\end{cases}
$$

As you see the upper bound obtained from Theorem \ref{F10} for $\hat{R}(C_{n},C_{n})$ is slightly better than the one obtained from Theorem \ref{F9}. In general, it unknown for us that which of these theorems gives the better upper bound for the size Ramsey number of $ t $ cycles.

We close the paper by the examination of the third and final random structure model. 
When all the cycles are even, using the binomial random bipartite graphs, we can obtain an upper bound with the same order of magnitude as the one in Theorem~\ref{F9}, but with a better constant factor. The \emph{binomial random bipartite graph} $\G(n,n, p)$ is the random bipartite graph $G=(V_1 \cup V_2, E)$ with the partite sets $V_1, V_2$, each of order $n$, in which every pair $\{i,j\} \in V_1 \times V_2$ appears independently as an edge in $G$ with probability~$p$.
Note that $p=p(n)$ may (and usually does) depend on $ n $.
The following is the counterpart of Lemmas~\ref{lem:random1} and \ref{thm:d-reg1} for the random bipartite graphs.

\begin{lemma}\label{lem:random2}
Let $c \in \R_+$ and let $d=d(c)$ be such that
$$2(1-c)\ln(1-c) + 2c \ln(c)+c^2d\geq 0.$$
Then, a.a.s. for every two sets of vertices $S$ and $T$ in different color classes of $G \in \G(N,N,d/N)$ with $|S|=|T|=cN$,  we have $e(S,T) \geq 1$.
\end{lemma}
\begin{proof}
Let $S$ and $T$ with $|S| = |T| = cN$ be fixed and let $X=X_{S,T}=e(S,T)$.
Clearly, $\E X = c^{2}dN> 0$ and by the Chernoff bound
\begin{eqnarray*}
\Pr(X = 0)=\Pr(X \le 0) &\le& \exp\left( - \E X  \right)=\exp \Big( -c^{2}dN  \Big).
\end{eqnarray*}
Thus, by the union bound over all choices of $S$ and $T$ we have
\begin{align*}
\Pr\left( \bigcup_{S,T} ( X_{S,T} = 0 ) \right) &\le
\binom{N}{cN} ^2 \exp \Big( -c^{2}dN \Big) \\
&= (\frac{ N!}{(cN)!((1-c)N)!})^2 \exp \Big( -c^{2}dN \Big).
\end{align*}
Using Stirling's formula we get
\begin{align*}
\Pr & \left( \bigcup_{S,T} (X_{S,T} = 0) \right)\leq \frac{1}{2\pi c(1-c)N}.\left(\frac{\exp \Big( -c^{2}d/2  \Big)}{c^{c}(1-c)^{1-c}}\right) ^{2 N}\\
& \leq  \frac{1}{2\pi c(1-c)N}=o(1),
\end{align*}
as desired.
\end{proof}

The following theorem gives an upper bound for the size Ramsey number of long even cycles. The bound has the same order of magnitude as the one in Theorem~\ref{F9}, but has a better constant factor.

\begin{theorem}\label{F11}
Assume that $n_1,\ldots,n_t$ are even positive integers and $ n=\max\{n_1,\ldots, n_t\} $. Also, suppose that for each $1\leq i\leq t$ we have $n_i\geq 2\lceil \log (81^tn) \rceil+2$.
Then for sufficiently large $n$, we have
$$\hat{R}(C_{n_1},\ldots,C_{n_t})\leq 2\times 81^{2t} (t\ln 81+1)\, n.$$
\end{theorem}

\begin{proof}
Let $ c=81^t $, $ N=c n$, $ d= (-2(1-c^{-1})\ln(1-c^{-1}) - 2c^{-1} \ln(c^{-1}))/c^{-2}  $ and $G=\G(N,N,d/N)$. By Lemma \ref{lem:random2}, a.a.s. for every two sets of vertices $S$ and $T$ in different color classes of $G$ with $|S|=|T|=n$, we have $e(S,T) \geq 1$. Therefore a.a.s. the complement of $G$ with respect to $K_{N,N}$ does not contain $K_{n,n}$ as a subgraph and so by Corollary \ref{F0}, we have $G\longrightarrow (C_{n_1},\ldots,C_{n_t})$.
On the other hand, the expected number of edges of $G$ is  $Nd$ and concentration around the expectation follows immediately from the Chernoff bound. This means that for sufficiently large $n$ we have
$$\hat{R}(C_{n_1},\ldots,C_{n_t})\leq Nd= 2c^2(c\ln c- (c-1)\ln(c-1))\, n\leq 2c^2 (\ln c+1)\, n,$$
where the last inequality is due to the mean value theorem.
\end{proof}
As a consequence of Theorem~\ref{F11}, for sufficiently large even $n$, we have
$$\hat{R}(C_{n},C_{n}) \leq 843\times 10^{6}n.$$


\end{document}